\documentclass[12pt,reqno,a4paper]{amsart}
\usepackage{graphicx,caption,amssymb,bbm}
\usepackage{hyperref}

\DeclareFontFamily{U}{mathb}{\hyphenchar\font45}
\DeclareFontShape{U}{mathb}{m}{n}{
      <5> <6> <7> <8> <9> <10> gen * mathb
      <10.95> mathb10 <12> <14.4> <17.28> <20.74> <24.88> mathb12
}{}
\DeclareSymbolFont{mathb}{U}{mathb}{m}{n}
\DeclareMathSymbol{\llcurly}{3}{mathb}{"CE}
\DeclareMathSymbol{\ggcurly}{\mathrel}{mathb}{"CF}

\begin{document}

\let\kappa=\varkappa
\let\eps=\varepsilon
\let\phi=\varphi
\let\p\partial
\let\lle=\preccurlyeq
\let\ulle=\curlyeqprec

\def\Z{\mathbb Z}
\def\R{\mathbb R}
\def\C{\mathbb C}
\def\Q{\mathbb Q}
\def\P{\mathbb P}
\def\HH{\mathsf{H}}
\def\XX{\mathcal{X}}

\def\bbk{\mathbbm{k}}

\def\conj{\overline}
\def\Beta{\mathrm{B}}
\def\const{\mathrm{const}}
\def\ov{\overline}
\def\wt{\widetilde}
\def\wh{\widehat}

\renewcommand{\Im}{\mathop{\mathrm{Im}}\nolimits}
\renewcommand{\Re}{\mathop{\mathrm{Re}}\nolimits}
\newcommand{\codim}{\mathop{\mathrm{codim}}\nolimits}
\newcommand{\Aut}{\mathop{\mathrm{Aut}}\nolimits}
\newcommand{\lk}{\mathop{\mathrm{lk}}\nolimits}
\newcommand{\sign}{\mathop{\mathrm{sign}}\nolimits}
\newcommand{\rk}{\mathop{\mathrm{rk}}\nolimits}

\def\id{\mathrm{id}}
\def\Leg{\mathrm{Leg}}
\def\Jet{{\mathcal J}}
\def\sS{{\mathcal S}}
\def\lcan{\lambda_{\mathrm{can}}}
\def\ocan{\omega_{\mathrm{can}}}
\def\bgamma{\boldsymbol{\gamma}}

\renewcommand{\mod}{\mathrel{\mathrm{mod}}}

\newtheorem{mainthm}{Theorem}
\renewcommand{\themainthm}{{\Alph{mainthm}}}
\newtheorem{thm}{Theorem}[section]
\newtheorem{lem}[thm]{Lemma}
\newtheorem{prop}[thm]{Proposition}
\newtheorem{cor}[thm]{Corollary}

\theoremstyle{definition}
\newtheorem{exm}[thm]{Example}
\newtheorem{rem}[thm]{Remark}
\newtheorem{df}[thm]{Definition}
\newtheorem{que}[thm]{Question}
\newtheorem{conje}[thm]{Conjecture}

\numberwithin{equation}{section}
\newcommand{\aftersubsec}{\hfill\nopagebreak\par\smallskip\noindent}

\title{Legendrian links and d\'ej\`a vu moments}
\author[Nemirovski]{Stefan Nemirovski}
\thanks{This work was partially supported by the SFB TRR 191 {\it Symplectic Structures in Geometry, Algebra and
Dynamics}, funded by the DFG (Projektnummer 281071066~--~TRR 191).}
\address{%
Steklov Mathematical Institute, Gubkina 8, 119991 Moscow, Russia;\hfill\break
\phantom{\& }Fakult\"at f\"ur Mathematik, Ruhr-Universit\"at Bochum, 44780 Bochum, \phantom{\& }Germany}
\email{stefan@mi-ras.ru}

\begin{abstract}
A Legendrian link is called a d\'ej\`a vu link if its components can be connected
by a positive Legendrian isotopy but this isotopy cannot be embedded. This is the
contact geometric analogue of a pair of events in a spacetime such that there are
d\'ej\`a vu moments on every future-directed timelike path between them. 
We construct d\'ej\`a vu links in several geometrically relevant situations
and discuss their basic properties.
\end{abstract}

\maketitle

\section{Introduction and overview}
Let $(Y,\xi)$ be a contact manifold with a co-oriented contact structure.
An isotopy of Legendrian submanifolds $\iota:L\times[0,1]\to Y$
is said to be {\it embedded\/} if the map $\iota$ is an embedding 
and {\it positive\/} if the trajectories of points on $L$ 
are positively transverse to the co-oriented contact distribution, 
see \S\ref{LegIso} for a discussion of different types of 
Legendrian isotopies. Every positive isotopy is embedded locally. 
A basic example of a positive isotopy is given by the action of the Reeb flow 
of a contact form defining~$\xi$. This isotopy is embedded if 
there are no Reeb chords of action $\le 1$ for $L_0=\iota(L\times\{0\})$.

For the purposes of this paper, a {\it Legendrian link\/} is 
an ordered pair $(\Lambda_1,\Lambda_2)$ of disjoint closed 
connected Legendrian submanifolds in~$Y$.

\begin{df}
A Legendrian link $(\Lambda_1,\Lambda_2)$ is called a {\it d\'ej\`a vu link\/} if
it satisfies the following two conditions:
\begin{itemize}
\item[(P)] There is a positive isotopy connecting $\Lambda_1$ to $\Lambda_2$. 
\item[(DjV)] There is {\it no\/} embedded positive isotopy connecting $\Lambda_1$ to $\Lambda_2$.
\end{itemize}
\end{df}

The definition and terminology are motivated by the connection with
Lorentz geometry recalled in \S\ref{spacetimes}. The {\it sky\/}  
(or {\it celestial sphere\/}) of a point in a reasonable spacetime~$\XX$
is a Legendrian sphere in the contact manifold of light rays of~$\XX$.
Future-directed timelike curves in $\XX$ induce positive Legendrian
isotopies of skies. If the skies
of two points in $\XX$ form a d\'ej\`a vu link, then every future-directed
timelike curve connecting these points contains {\it d\'ej\`a vu moments}, 
that is, pairs of distinct points lying on the same light ray, 
see \S\ref{djvLorentz}. 

If the Legendrian isotopy class of the components $\Lambda_1$ and $\Lambda_2$ is fixed, 
there is only one Legendrian isotopy class of {\it non\/} d\'ej\`a vu links satisfying
condition~(P) by Corollary~\ref{NonDjV}. Legendrian links satisfying~(P)
and not smoothly isotopic to links in this `primitive' class are clearly d\'ej\`a vu. 
Proposition~\ref{DjVS1Top} shows that this observation can lead to 
a topological description of d\'ej\`a vu links in certain cases 
but being d\'ej\`a vu is not a topological property in general, 
see Example~\ref{SmVsLeg}. 

The components of a d\'ej\`a vu link may sometimes be connected 
by an embedded Legendrian isotopy which is not positive, see \S\S\ref{yxl} 
and~\ref{DjVS1}. It seems unlikely, however, that this could happen
if the Legendrian isotopy class of the components of the link
is orderable. A partial result in this direction is obtained 
in Proposition~\ref{DejavuOrdNoEmb} based on a partial answer
to Question~\ref{QMonNonneg}. 

D\'ej\`a vu links with components in the most basic orderable Legendrian
isotopy class of the zero section of a one-jet bundle are constructed
in~\S\ref{EmbJet} using generating function methods recalled briefly 
in~\S\ref{QIF}. It is a simple example of a problem in which one must use 
the `non-persistent' finite bars of the barcode of a generating 
function rather than the infinite bars and the associated spectral
invariants.

\section{Legendrian submanifolds}

\subsection{Legendrian isotopies}
\label{LegIso}
\aftersubsec
A parametrised Legendrian isotopy in a contact manifold $(Y,\xi)$
is a smooth map
$$
\iota:L\times [0,1] \longrightarrow Y
$$
such that $\iota|_{L\times\{t\}}:L\times\{t\} \hookrightarrow Y$ is an embedding
and the submanifold $\Lambda_t=\iota(L\times\{t\})$ is Legendrian for all $t\in [0,1]$.
Two parametrised isotopies are equivalent if they differ by 
a diffeomorphism $\phi$ of $L\times [0,1]$ such that $\phi(L\times \{t\})=L\times \{t\}$
for all $t\in [0,1]$. A Legendrian isotopy is an equivalence class of parametrised 
Legendrian isotopies. 

By the Legendrian isotopy extension theorem~\cite[Theorem~2.6.2]{Ge},
a~Legendrian isotopy of closed Legendrian submanifolds can be extended
to a compactly supported contact isotopy of the ambient contact
manifold. In particular, Legendrian isotopic Legendrian links are ambiently
contactomorphic.

Let us assume henceforth that Legendrian submanifolds are closed 
(i.e.\ compact and without boundary) and connected.

A Legendrian isotopy parametrised by $\iota:L\times [0,1] \longrightarrow Y$ 
in a co-oriented contact manifold $(Y,\xi=\ker\alpha)$ is called 
{\it non-negative\/} if $\iota^*\alpha(\frac{\p}{\p t})\ge 0$ on $L\times [0,1]$.
If the inequality is everywhere strict, the isotopy is called {\it positive}.
Both properties do not depend on the parametrisation and on the choice
of a contact form defining the co-oriented contact structure;
they are also invariant under (co-orientation preserving) contactomorphisms.

We write $\Lambda\lle\Lambda'$ if there is a non-negative Legendrian isotopy 
from $\Lambda$ to $\Lambda'$ and $\Lambda\llcurly\Lambda'$ if
there is a positive one. The relation $\lle$ is clearly reflexive and transitive;
if it is also antisymmetric on a Legendrian isotopy class $\mathcal{L}$,
then this class is called {\it orderable}. By~\cite[Proposition 4.7]{ChNe3},
$\mathcal{L}$ is orderable if and only if it does not contain a positive 
Legendrian loop, i.e.\ if and only if $\llcurly$ is {\it not\/} reflexive on~$\mathcal{L}$.

A Legendrian isotopy $\{\Lambda_t\}_{t\in [0,1]}$ is called $\lle$-{\it monotone\/}
if $\Lambda_{t_1}\lle \Lambda_{t_2}$ for all $0\le t_1<t_2\le 1$. 

\begin{lem}
\label{EmbIncrease}
A Legendrian isotopy $\{\Lambda_t\}_{t\in [0,1]}$ such that $\Lambda_0\lle \Lambda_1$ 
and $\Lambda_{t_1}\cap\Lambda_{t_2}=\varnothing$ for all $t_1\ne t_2$
is $\lle$-monotone.
\end{lem}

\begin{proof}
The Legendrian links $(\Lambda_{t_1},\Lambda_{t_2})$ are Legendrian
isotopic and hence contactomorphic for all $0\le t_1<t_2\le 1$. 
The result follows because $\lle$ is preserved by contactomorphisms.
\end{proof}

\begin{rem}
Recall from \cite[Lemma 2.2]{ChNe2} that $\llcurly$ is equivalent to $\lle$ for disjoint $\Lambda_t$'s.
Hence, the isotopy in the lemma is in fact $\llcurly$-monotone.
\end{rem}

A non-negative isotopy is obviously $\lle$-monotone but the converse
is not true in general.

\begin{exm}
\label{EmbNonOrd}
Let $\mathcal{L}$ be a {\it non\/}orderable Legendrian isotopy class
(e.g.\ any class containing a {\it loose\/} Legendrian~\cite{Li}
or any class in a contact manifold admitting a periodic Reeb flow). 
Then there is a positive Legendrian loop based at every $\Lambda\in\mathcal{L}$, 
which implies that $\Lambda\llcurly\Lambda'$
for every $\Lambda'$ sufficiently $C^1$-close to $\Lambda$,
cf.\ \cite[Proof of Corollary 8.1]{ChNe2}. If now $\{\Lambda_t\}_{t\in [0,1]}$
is any Legendrian isotopy in $\mathcal{L}$ such that $\Lambda_t$ are pairwise
disjoint, then it is $\llcurly$-monotone by the proof of Lemma~\ref{EmbIncrease}.
\end{exm}

\begin{que}
\label{QMonNonneg}
Are $\lle$-monotone isotopies non-negative in every {\it orderable\/} Legendrian isotopy class?
\end{que}

For the Legendrian isotopy class of the zero section of the $1$-jet bundle~$\Jet^1(L)$ 
of a closed manifold~$L$, the positive answer to the above question 
follows easily from~\cite[Corollary 5.4]{ChNe1}, 
which is essentially equivalent to the orderability 
of that class~\cite[Corollary 5.5]{ChNe1}. There is 
another case in which we are now going to show that $\lle$-monotone 
Legendrian isotopies are non-negative.

\begin{lem}
\label{SphMonNonneg}
Let $\mathcal{L}$ be an orderable Legendrian isotopy class of spheres.
A $\lle$-monotone Legendrian isotopy $\{\Lambda_t\}_{t\in [0,1]}\subset\mathcal{L}$  
is non-negative.
\end{lem}

\begin{proof}
Suppose that the isotopy is not non-negative at some $\tau\in [0,1]$ and denote $\Lambda=\Lambda_\tau$.
Fix a contactomorphism $\Psi$ from a neighbourhood $U\supset\Lambda$ to a tubular neighbourhood of the zero section 
in $\Jet^1(\Lambda)$ mapping $\Lambda$ onto the zero section~$\mathrm{O}$, see~\cite[Example 2.5.11]{Ge}. 
For $t$ close enough to $\tau$, the Legendrian
$\Lambda_t$ corresponds to the graph of the $1$-jet of a smooth function $f_t:\Lambda\to\R$.
We may now assume that $f_t$ is $C^1$-small and negative somewhere on $\Lambda$
but $\Lambda\lle \Lambda_t= \Psi^{-1}(j^1(f_t))$. 

Let $F$ be a smooth function on $\Lambda$ such that 
\begin{itemize}
\item[a)] $F\ge f_t$;
\item[b)] $\{F<0\}$ is a ball in $\Lambda$;
\item[c)] zero is not a critical value of $F$ (i.e.\ $j^1(F)\cap \mathrm{O}=\varnothing$);
\item[d)] $j^1(F)\subset \Psi(U)$.
\end{itemize}
$F$ may be defined as a regularised maximum (see \cite[Lemma I.5.18]{De}) 
of $f_t$ and a $C^1$-small function $\phi$ such that $\{\phi\le 0\}$ is a closed ball 
contained in~$\{f_t<0\}\ne\varnothing$.

Property (a) implies that $\Lambda\lle \Psi^{-1}(j^1(f_t))\lle \Psi^{-1}(j^1(F))$. 
The space of functions satisfying (b)--(d) is connected for any $\Lambda$. 
If $\Lambda\cong S^n$, the function $-F$ satisfies (b)--(d) too.
Hence, the links $(\mathrm{O},j^1(F))$ and $(\mathrm{O},j^1(-F))$
are Legendrian isotopic in $\Psi(U)$. The latter link is Legendrian
isotopic to $(j^1(F),\mathrm{O})$ via the `shift' contact isotopy
$$
(q,p,u)\mapsto \left(q,p+s\tfrac{\p F}{\p q}, u+sF(q)\right), \quad s\in [0,1].
$$
Thus, $(\Lambda,\Psi^{-1}(j^1(F)))$ is Legendrian isotopic to $(\Psi^{-1}(j^1(F)),\Lambda)$
and therefore we have both $\Lambda\lle \Psi^{-1}(j^1(F))$ and $\Psi^{-1}(j^1(F))\lle\Lambda$, 
which contradicts the assumption that $\mathcal{L}$ is orderable.
\end{proof}

\begin{rem}[Irreversible Legendrian links]
The key point in the above proof will not work 
if $\Lambda$ is not a homology sphere, i.e.\ if there
exists a non-zero homology class $\beta\in\HH_k(\Lambda;\bbk)$
of degree $k\ne 0, \dim\Lambda$. Namely, the links
$(\mathrm{O},j^1(F))$ and $(j^1(F),\mathrm{O})$
will {\it not\/} be Legendrian isotopic in $\Jet^1(\Lambda)$
for any function $F$ satisfying (b) and~(c).
To see this, observe that if $S_F$ is any 
quadratic at infinity generating function for $j^1(F)$
(see \S\ref{QIF}), then $c_\beta(S_F)=c_\beta(F)>0$
and $c_\beta(-S_F)=c_\beta(-F)<0$, where $c_\beta$ is the spectral invariant
defined in Remark~\ref{spectr}. Hence, one can apply
Traynor's argument from~\cite[\S 5]{Tr}
with $c_\beta$ instead of $c_+=c_{[\Lambda]}$.
\end{rem}

\begin{rem}[Lorentzian comparison]
The statement analogous to Lemma~\ref{SphMonNonneg} in Lorentz geometry
(in the sense explained in~\S\ref{spacetimes}) asserts that a causally
monotone curve in a {\it distinguishing\/} spacetime is future-directed,
see e.g.~\cite[Proposition 3.19]{MS}. Note that orderability 
is formally analogous to causality, so the Lorentzian statement 
requires a stronger assumption.
\end{rem}

A Legendrian isotopy parametrised by $\iota:L\times [0,1] \longrightarrow Y$ 
is called {\it immersed\/} or {\it embedded\/} if the map $\iota$ is an immersion
or an embedding. Clearly, an immersed isotopy is embedded if and only if $\iota$ 
is injective. These notions are obviously independent of the choice of
a parametrisation. The property of a Legendrian isotopy to be {\it immersed\/}
can also be expressed in terms of the pull-back of a contact form.

\begin{lem}
\label{Immersed}
A Legendrian isotopy $\iota:L\times [0,1]\to (Y, \xi=\ker\alpha)$
is \textbf{not\/} immersed at a point $(q,\tau)$ if and only if $(q,\tau)$ is a
critical point of the function $\iota^*\alpha(\frac{\p}{\p t})$ 
restricted to $L\times\{\tau\}$ with critical value zero.
\end{lem}

\begin{proof}
The condition $\iota^*\alpha(\frac{\p}{\p t})=0$ is equivalent to $\iota_*\frac{\p}{\p t}\in\xi_{\iota(q,\tau)}$.
The point $(q,\tau)$ is critical for $\iota^*\alpha(\frac{\p}{\p t})$ on $L\times\{\tau\}$ if and only if
$$
X\left(\iota^*\alpha(\tfrac{\p}{\p t})\right) = 0
$$
at $(q,\tau)$ for all {\it vertical\/} vector fields $X$ on~$L\times [0,1]$.
For a vertical~$X$, the commutator $[X,\frac{\p}{\p t}]$
is also vertical and $\iota^*\alpha(X)=\alpha(\iota_*X)=0$. Hence,  
$$
\begin{array}{rcl}
X\left(\iota^*\alpha(\tfrac{\p}{\p t})\right)&=& 
d(\iota^*\alpha)\left(X,\tfrac{\p}{\p t}\right) + \tfrac{\p}{\p t}(\iota^*\alpha(X)) + \iota^*\alpha\left([X,\tfrac{\p}{\p t}]\right)\\[2pt]
&=& d(\iota^*\alpha)\left(X,\tfrac{\p}{\p t}\right)\\[2pt]
&=& d\alpha\left(\iota_*X,\iota_*\tfrac{\p}{\p t}\right).
\end{array}
$$
Thus, our assumptions are equivalent to $\iota_*\frac{\p}{\p t}$ 
being skew-orthogonal to the Lagrangian subspace $\iota_*(T_q L)$ in the symplectic
vector space $(\xi_{\iota(q,\tau)}, d\alpha)$. This means precisely 
that $\iota_*\frac{\p}{\p t}\in \iota_*(T_q L)$  and the rank 
of $\iota$ is not maximal at~$(q,\tau)$.
\end{proof}

As an application, we show that being immersed characterises positive 
Legendrian isotopies among non-negative ones.

\begin{cor}
\label{NonnegEmbPos}
A non-negative Legendrian isotopy is positive if and only if it is immersed.
\end{cor}

\begin{proof}
A positive Legendrian isotopy is obviously immersed. To prove the `if' part,
let $\iota:L\times [0,1]\to (Y, \xi=\ker\alpha)$ be a parametrisation 
of an immersed non-negative isotopy. Suppose that the isotopy isn't positive. 
Then $\iota^*\alpha(\frac{\p}{\p t})=0$ 
at some point $(q,\tau)\in L\times[0,1]$.
Since $\iota^*\alpha(\frac{\p}{\p t})\ge 0$, it follows that this function
attains its minimum equal to zero at $(q,\tau)$. Hence, the isotopy 
is not immersed at that point by Lemma~\ref{Immersed}.
\end{proof}

\begin{rem} 
A related argument may be found in~\cite[Lemma~4.12(i)]{GKS}.
\end{rem}

The existence of an embedded isotopy between two Legendrians
may be inferred from a seemingly weaker assumption.

\begin{lem}
\label{ExistEmb}
If $\Lambda_0\cap\Lambda_t=\varnothing$ for all $t>0$
in a Legendrian isotopy $\{\Lambda_t\}_{t\in [0,1]}$,
then $\Lambda_0$ and $\Lambda_1$ are connected
by an embedded Legendrian isotopy. 
If $\{\Lambda_t\}_{t\in [0,1]}$ is also positive at $t=0$,
then $\Lambda_0$ and $\Lambda_1$ are connected
by a positive embedded Legendrian isotopy.
\end{lem}

\begin{proof}
If we identify a neighbourhood
of $\Lambda_0$ with a neighbourhood of the zero section 
in $\Jet^1(\Lambda_0)$, then $\Lambda_\tau$ for small enough $\tau>0$ 
corresponds to $j^1(f_\tau)$ for a smooth function $f_\tau$ on $\Lambda_0$
such that zero is not its critical value. Hence, an embedded
isotopy connecting $\Lambda_0$ to $\Lambda_\tau$ can be defined 
by $\{j^1(sf_\tau)\}_{s\in [0,1]}$. If the given isotopy 
is positive at $t=0$, then it is positive and embedded on $[0,\tau]$
for small enough~$\tau>0$. 
(This corresponds to $f_{\tau'}>f_{\tau''}>0$ on $\Lambda_0$ for $\tau\ge\tau'>\tau''>0$.)
It remains to observe that $\Lambda_\tau$ is Legendrian isotopic
to $\Lambda_1$ in the complement of $\Lambda_0$ and therefore
there exists a contact isotopy $\{\phi_t\}_{t\in [\tau,1]}$ 
such that $\Lambda_0\cap\mathop{\mathrm{supp}}\phi_t =\varnothing$
and $\phi_t(\Lambda_\tau)=\Lambda_t$. Applying $\phi_1$
to the Legendrian isotopies from $\Lambda_0$ to $\Lambda_\tau$ 
constructed above completes the proof.
\end{proof}

\begin{rem}
\label{EmbHomotopic}
The embedded isotopy obtained in the lemma is homotopic
to the original Legendrian isotopy through Legendrian
isotopies. If the original isotopy is positive,
this homotopy is through positive isotopies.
\end{rem}

Since the space of positive functions on a manifold is
connected, the proof of Lemma~\ref{ExistEmb} implies 
the following result.

\begin{cor}
\label{NonDjV}
All \textbf{non} d\'ej\`a vu Legendrian links $(\Lambda,\Lambda')$
with $\Lambda\llcurly\Lambda'$ and a given $\Lambda$
are Legendrian isotopic by an isotopy fixing $\Lambda$.
\end{cor}

It follows that the Legendrian isotopy class of a non d\'ej\`a vu link
with $\Lambda\llcurly\Lambda'$ is completely determined by the Legendrian
isotopy class of its components. One way of representing this class
of links is to take the link formed by a Legendrian and its sufficiently small
shift along the Reeb flow of a contact form. 

\begin{prop}
\label{DejavuOrdNoEmb}
Suppose that $\mathcal{L}$ is either the Legendrian isotopy class 
of the zero section in $\Jet^1(L)$ or an orderable Legendrian
isotopy class of spheres. Let $(\Lambda,\Lambda')$ be 
a Legendrian link with components in~$\mathcal{L}$.
\begin{itemize}
\item[(i)] 
If $\Lambda\llcurly\Lambda'$, then every embedded isotopy
from $\Lambda$ to $\Lambda'$ is positive.
\item[(ii)]
If $(\Lambda,\Lambda')$ is a d\'ej\`a vu Legendrian link,
then for every Legendrian isotopy $\{\Lambda_t\}_{t\in [0,1]}$ with
$\Lambda_0=\Lambda$ and $\Lambda_1=\Lambda'$ there exists
a $t_0>0$ such that $\Lambda\cap\Lambda_{t_0}\neq\varnothing$
and a $t_1<1$ such that $\Lambda'\cap\Lambda_{t_1}\neq\varnothing$.
\end{itemize}
\end{prop}

\begin{proof}
(i) An embedded Legendrian isotopy from $\Lambda$ to $\Lambda'$ is $\lle$-monotone by Lemma~\ref{EmbIncrease}.
Any $\lle$-monotone isotopy in~$\mathcal{L}$ is non-negative 
by Lemma~\ref{SphMonNonneg} and the discussion preceding it. 
Finally, every embedded non-negative isotopy is positive by Corollary~\ref{NonnegEmbPos}.

\smallskip
\noindent
(ii) Assume that there is a Legendrian isotopy connecting $\Lambda$ to $\Lambda'$
which is disjoint from $\Lambda$ for all $t>0$ or 
from $\Lambda'$ for all $t<1$. Then by Lemma~\ref{ExistEmb} 
there is an embedded isotopy connecting $\Lambda$ to $\Lambda'$.
Applying (i) we obtain a contradiction with the definition of a d\'ej\`a vu link.
\end{proof}

\begin{rem}
Proposition~\ref{DejavuOrdNoEmb} holds true for every Legendrian isotopy class
in which $\lle$-monotone isotopies are non-negative. So it will hold for every 
orderable Legendrian isotopy class if Question~\ref{QMonNonneg} can be
answered in the positive. 
\end{rem}

\begin{rem} 
\label{UODejavuOrdNoEmb}
A version of Proposition~\ref{DejavuOrdNoEmb} 
can be formulated for {\it universally orderable\/}
Legendrian isotopy classes, see~\cite[\S 4.3]{ChNe3}. Namely,
one has to assume that the Legendrian isotopies in (i) and~(ii)
are {\it a~priori\/} homotopic through Legendrian isotopies 
to positive isotopies. A careful inspection of the proofs in this section  
shows that the modified statement is true for a universally orderable 
class of Legendrian spheres. On the other hand, the result can be false 
without this additional assumption, see~\S\ref{yxl} and~\S\ref{DjVS1}.
\end{rem}

\subsection{Quadratic at infinity functions and Legendrians}
\label{QIF}
\aftersubsec
Let $L$ be a closed connected manifold. A smooth function 
$$
S=S(q,\xi):L\times\R^N\to\R
$$ 
is said to be {\it quadratic at infinity\/} if 
$$
S(q,\xi)=\sigma(q,\xi) + Q(\xi),
$$
where $\sigma$ has compact support in $L\times \R^N$ 
and $Q$ is a non-degenerate quadratic 
form on~$\R^N_\xi$. Denote by
$$
S^c:=\{(q,\xi)\in L\times\R^N\mid S(q,\xi)< c\}
$$
the sublevel sets of $S$ and let $S^{-\infty}$ be the set $S^c$ 
for a sufficiently negative~$c\ll 0$. The following standard lemma 
is an immediate consequence of the isotopy extension theorem.

\begin{lem}
\label{stability}
Let $S_t = \sigma_t + Q$, $t\in [0,1]$ be a smooth family of quadratic at infinity
functions. Suppose that $a\in\R$ is not a critical value of $S_t$ for all~$t$.
Then the inclusions $\imath^a_t:S^a_t\hookrightarrow L\times\R^N$ are
isotopic by a compactly supported isotopy constant on $S_t^{-\infty}$. 
In particular, the relative homology groups $\HH_*(S_t^a,S_t^{-\infty};\bbk)$ 
are isomorphic for all~$t$ and the induced homomorphisms
$$
\left(\imath^a_t\right)_*: \HH_*(S_t^a,S_t^{-\infty};\bbk) \longrightarrow \HH_*(L\times\R^N,S_t^{-\infty};\bbk)
$$
do not depend on~$t$.
\end{lem}

\begin{rem}[Persistence and barcodes]
\label{pers}
If $S$ is a Morse function (i.e.\ its critical points are non-degenerate), 
then the relative homology groups $\HH_*(S^c,S^{-\infty};\bbk)$, $c\in\R$, 
together with the homomorphisms induced by the inclusions 
$\imath^{c,c'}:S^c\hookrightarrow S^{c'}$ for $c\le c'$
form a {\it persistence module\/} in the sense of \cite[Definition~1.1.1]{PRSZ}.
Lemma~\ref{stability} implies that if two quadratic at infinity Morse functions 
can be connected by a family such that $a\in\R$ is never a critical value, then 
in the {\it barcodes\/} associated to their persistence modules 
by \cite[Theorem 2.1.1]{PRSZ} the number of bars containing $a$ is the same.
\end{rem}

\begin{rem}[Spectral invariants]
\label{spectr}
Let $\R^N=V_+\times V_-$ be a decomposition into linear subspaces such that $Q$ 
is positive definite on~$V_+$ and negative definite on~$V_-$. The dimension
$\nu=\dim V_-$ is the (negative) index of~$Q$. The map
$$
\HH_*(L;\bbk) \ni \beta \longmapsto \beta\times [V_-]\in\HH_{*+\nu}(L\times\R^N,S^{-\infty};\bbk)
$$
is an isomorphism between the homology of $L$ and the shifted homology
of $L\times\R^N$ relative to~$S^{-\infty}$. Following Viterbo~\cite[\S2]{Vi}, 
one can therefore associate critical values of~$S$ to homology classes on~$L$.
Namely, for $\beta\in \HH_k(L;\bbk)$, define
$$
c_\beta(S):=\inf\bigl\{c\in\R\mid \beta\times [V_-]\in \left(\imath^c\right)_*\HH_{k+\nu}(S^c,S^{-\infty};\bbk)\bigr\},
$$
where $\imath^c:S^c\to L\times\R^N$ is the inclusion. This is a critical value of $S$
by basic Morse theory. Moreover, $c_\beta(S_t)$ is a continuous function of $t$ 
for any smooth family of quadratic at infinity functions, cf.~\cite[Proposition 2.5]{Vi}.
In the barcode associated to $S$, the numbers $c_\beta(S)$ are precisely the endpoints
of the {\it infinite\/} bars.
\end{rem}

Let $\Jet^1(L)$ denote the $1$-jet bundle of~$L$ equipped with 
the standard contact form $du-\lcan$,
where $u$ is the fibre coordinate in $\Jet^0(L)$ and $\lcan=p\,dq$
is the Liouville form on $T^*L$.

\begin{df} A smooth function 
$$
S=S(q,\xi):L\times\R^N\to\R
$$ 
is  a {\it generating function\/} for a Legendrian submanifold $\Lambda\subset\Jet^1(L)$ 
if zero is a regular value of the partial differential $d_\xi S$ and the map
\begin{equation}
\label{gendef}
\{d_\xi S(q,\xi)=0\}\ni (q,\xi) \longmapsto (q,d_q S(q,\xi),S(q,\xi))\in \Jet^1(L)
\end{equation}
is a diffeomorphism onto~$\Lambda$. 
\end{df}

Note that (Morse) critical points of $S$
as a function of both variables $(q,\xi)$ are in one-to-one 
correspondence with (transverse) intersection points of $\Lambda$
with $\{p=0\}\subset\Jet^1(L)$. In particular, the intersection
of $\Lambda$ with the zero section is parametrised by the
critical points of $S$ with critical value zero.

By Chekanov's theorem~\cite{Che}, for any Legendrian isotopy  
$\{\Lambda_t\}_{t\in [0,1]}$ such that $\Lambda_0$ is the zero section $\mathrm{O}\subset\Jet^1(L)$, 
there exists a smooth family of {\it quadratic at infinity\/} 
generating functions 
$$
S_t = \sigma_t + Q :L\times\R^N\to\R
$$ 
for $\Lambda_t$ with $\sigma_0\equiv 0$. 
Furthermore, this family is unique up to stabilisations and fibrewise
diffeomorphisms by the Viterbo--Th\'eret theorem~\cite{Vi, Th1, Th2}.

\section{Examples of d\'ej\`a vu links}

\subsection{D\'ej\`a vu \textit{versus\/} Wiedersehen}
\label{yxl}
\aftersubsec
Let $M$ be a connected manifold, $\dim M\ge 2$, 
and let $ST^*M$ be its co-sphere bundle with the canonical contact structure. 
The fibres of $ST^*M$ are Legendrian spheres and the
links formed by any two of them are Legendrian isotopic.

The link of two fibres of $ST^*M$ may satisfy condition (P) for some~$M$. Suppose, for instance, that there 
is a Riemannian metric on $M$ making it into a $Y^x_\ell$-manifold, 
i.e.\ such that all unit speed geodesics from $x\in M$ return to $x$ 
at time~$\ell>0$.
(Examples of $Y^x_\ell$-manifolds include compact rank one symmetric spaces and
certain exotic spheres, see~\cite[\S 7.C]{Be}.)
The co-geodesic flow of this metric defines a positive
Legendrian loop based at the fibre $ST_x^*M$.
Hence, there is a positive Legendrian isotopy between 
(any) two fibres, see~\cite[\S 8]{ChNe2}.

In the special case when $M$ is diffeomorphic to the sphere,
the link of two fibres satisfies (P) but is {\it not\/}
d\'ej\`a vu. Indeed, any two points $x\ne y$ on $M$
are antipodal for the pull-back of the
standard round metric by some diffeomorphism. 
The co-geodesic flow of this metric defines an 
{\it embedded\/} positive Legendrian isotopy 
from $ST_x^*M$ to $ST_y^*M$. 
(The term `Wiedersehen' refers to Riemannian manifolds 
on which the cut locus of every point is another point; 
the only such manifolds are the standard round spheres~\cite{Ka,Ya}.)

\begin{prop}
\label{BSnondjv}
Let $M$ be a connected manifold not homeomorphic to the sphere.
If there is a positive Legendrian isotopy connecting 
two different fibres of $ST^*M$, then they form 
a d\'ej\`a vu Legendrian link.
\end{prop}

\begin{proof}
Let us assume that there is an embedded positive Legendrian
isotopy from $ST_x^*M$ to $ST_y^*M$, $x\ne y\in M$, 
and prove that $M$ must be homeomorphic to the sphere.

As in the proof of Lemma~\ref{ExistEmb}, we may arrange that
the embedded isotopy has a standard form near its endpoints.
In the case of $ST^*M$, this standard form may be taken to be
the action of a co-geodesic flow. It follows that if 
$\pi:ST^*M\to M$ is the bundle projection and $\iota:S^{n-1}\times [0,1]\to ST^*M$
is a parametrisation of the isotopy, then the map $\pi\circ\iota: S^{n-1}\times (0,1)\to M$ 
extends to a smooth map $f:S^n\to M$, $n=\dim M$. Moreover, there is a contractible 
neighbourhood $U\ni x$ such that $f: f^{-1}(U)\to U$ is one-to-one.

The result now follows from a standard topological argument, cf. e.g.~\cite[\S 3.8]{McD}. 
Let $p:\wt{M}\to M$ be the universal covering, then $f$ lifts to a map $\wt{f}:S^n\to\wt{M}$.
There is a connected component $V$ of $\pi^{-1}(U)$ such that ${\wt{f}}^{-1}(V)=f^{-1}(U)$ 
and the map $\wt{f}$ is one-to-one over $V$. Therefore $\wt{f}$ 
has topological degree one. Since $S^n$ is compact, it follows that 
$\wt{f}$ is surjective onto $\wt{M}$. 
Hence, $p^{-1}(U)=\wt{f}(f^{-1}(U))=V$ and so $p$ is the trivial covering. 
This shows that $M=\wt{M}$ is simply connected and closed.
If $\alpha\in\HH^k(M;\bbk)$ is a non-zero cohomology class with $k\ne 0,n$ 
and coefficients in any field~$\bbk$,
then by Poincar\'e duality (on the closed orientable manifold~$M$) 
there is a class $\beta\in \HH^{n-k}(M;\bbk)$ such that $\alpha\cup\beta\ne 0\in \HH^n(M;\bbk)$. 
Then $0\ne f^*(\alpha\cup\beta)=f^*(\alpha)\cup f^*(\beta)$,
which is impossible since the cohomology of $S^n$ is trivial in degrees $k\ne 0,n$.
Thus the cohomology of $M$ with any (e.g.\ integer) coefficients 
vanishes in all degrees $k\ne 0, n$. This shows that $M$ is homeomorphic
to the sphere by the classification of surfaces for $n=2$
and by the theorems of Perelman, Freedman, and Smale for
$n=3, 4$, and~$\ge 5$, respectively.
\end{proof}

The existence of a positive Legendrian loop based at a fibre 
of $ST^*M$ implies that the universal cover $\wt{M}$ is 
compact~\cite{ChNe2} and the integral cohomology ring $\HH^*(\wt{M};\Z)$
is isomorphic to that of a CROSS~\cite{FrLaSch}. If the
loop does not intersect the fibre except at its basepoint, 
then $M$ is either simply connected or homotopy equivalent
to the real projective space $\R\mathrm{P}^n$ by~\cite{Da}.
The above proposition shows that the existence of an
embedded positive Legendrian isotopy between two fibres imposes
an even stronger topological condition on~$M$.
Note also that the proof does not use any tools from
contact topology such as generating functions or
Rabinowitz--Floer homology.

There are obvious embedded Legendrian isotopies
between two fibres of $ST^*M$ obtained by moving
the first fibre along an embedded curve
connecting the corresponding points in~$M$.
(Such isotopies are not positive or non-negative,
see~\cite[\S 6]{ChNe1}.) The Legendrian isotopy class 
of the fibre of $ST^*M$ is always universally orderable~\cite{ChNe3}
and so Remark~\ref{UODejavuOrdNoEmb} shows that 
those obvious isotopies are not homotopic to
a positive one.

\subsection{D\'ej\`a vu Legendrian links in~$T^*L\times S^1$ ({cf.~\cite[\S 8]{CS}})}
\label{DjVS1}
\aftersubsec
Let $Y:= T^*L\times S^1$ be the quotient of $\Jet^1(L)$ by the $\Z$-action 
generated by the shift~$(q,p,u)\mapsto (q,p,u+1)$ and denote by $\pi:\Jet^1(L)\to Y$ 
the projection to the quotient. The canonical contact form $du-\lcan$ is invariant under 
this action and descends to a contact form with periodic Reeb flow on~$Y$. 
In particular, every Legendrian isotopy class in $Y$ is nonorderable. 

Let $f:L\to\R$ be a smooth function such that 
\begin{itemize}
\item[1)] $\max_L |f|<\tfrac{1}{2}$;
\item[2)] zero is not a critical value of~$f$;
\item[3)] $f$ changes sign on $L$.
\end{itemize}
The first two conditions guarantee that $\pi(j^1(f))$ is a Legendrian submanifold of~$Y$ 
disjoint from $\pi(\mathrm{O})$.

\begin{prop}
$\bigl(\pi(\mathrm{O}),\pi(j^1(f))\bigr)$ is a d\'ej\`a vu link in~$Y$.
\end{prop}

\begin{proof}
First, taking the projection to $Y$ of the positive linear isotopy from $\mathrm{O}$ to $j^1(f+1)$ 
shows that $\pi(\mathrm{O})\llcurly\pi(j^1(f))$. 
Secondly, we need to prove that a positive isotopy from $\pi(\mathrm{O})$ to $\pi(j^1(f))$ 
cannot be embedded. Such an isotopy lifts to a positive isotopy 
from $\mathrm{O}$ to $j^1(f+k)$ for some $k\in\Z$. 
By \cite[Corollary~5.4]{ChNe1}, we have $f+k\ge 0$ on $L$ and therefore $k>0$. 
Consider now a family $S_t$, $t\in [0,1]$, of quadratic at infinity generating functions 
for the lifted isotopy and the spectral invariant $c_{[L]}$ associated 
to the fundamental class $[L]\in\HH_{\dim L}(L;\Z/2)$.
Then $c_{[L]}(S_0)=0$ and $c_{[L]}(S_1)=\max_L (f+k) >k$ so that
$c_{[L]}(S_\tau)=k$ for some $\tau\in (0,1)$ by the continuity of the spectral invariant.
Hence, $k$ is a critical value of $S_\tau$ or, in other words, the Legendrian generated 
by $S_\tau$ intersects $j^1(k)$. The projection of this Legendrian to $Y$ intersects 
$\pi(j^1(k))=\pi(\mathrm{O})$ and thus the positive isotopy from $\pi(\mathrm{O})$ 
to $\pi(j^1(f))$ in $Y$ is not embedded as claimed. 
\end{proof}

The linear isotopy from the zero section $\mathrm{O}$ to $j^1(f)$ projects 
to an {\it embedded\/} Legendrian isotopy from $\pi(\mathrm{O})$ 
to $\pi(j^1(f))$ in~$Y$, which shows that Proposition~\ref{DejavuOrdNoEmb} 
does not hold in this case. Note, however, that such an embedded isotopy
cannot be homotopic to a positive isotopy. (Otherwise it would lift to
an isotopy from $\mathrm{O}$ to $j^1(f+k)$ for $k>0$ and the last
part of the proof of the proposition would again lead to a contradiction.)
The Legendrian isotopy class of $\pi(\mathrm{O})$ is {\it universally\/}
orderable, so this agrees with Remark~\ref{UODejavuOrdNoEmb}.

\subsection{D\'ej\`a vu Legendrian links in $\Jet^1(S^1)$}
\label{EmbJetS1}
\aftersubsec
The starting point of the construction is the positive Legendrian
isotopy in $\Jet^1(\R)$ depicted on Fig.~\ref{PosIso}. The figure shows
the wavefronts, i.e.\ the projections of the Legendrians 
to the $(q,u)$-plane. The dashed line represents the zero section $\mathrm{O}$.
The solid wavefront is evolving by a Legendrian isotopy that is
positive because at every moment the points on the wavefront are
moving upwards with respect to the tangent line to the wavefront.
The Legendrian $\Lambda$ in (d) has a single transverse critical 
point (i.e.\ intersection with $\{p=0\}$) in $u\le 0$. 

Outside of a suitably chosen segment in $\R$ the isotopy coincides
with the linear isotopy from $\mathrm{O}$ to the $1$-jet of
a positive constant function. Therefore it can be completed to a
positive isotopy in $\Jet^1(S^1)$. 

\begin{figure}[ht]
\centering
\captionsetup{margin=-1cm}
\includegraphics[scale=0.6]{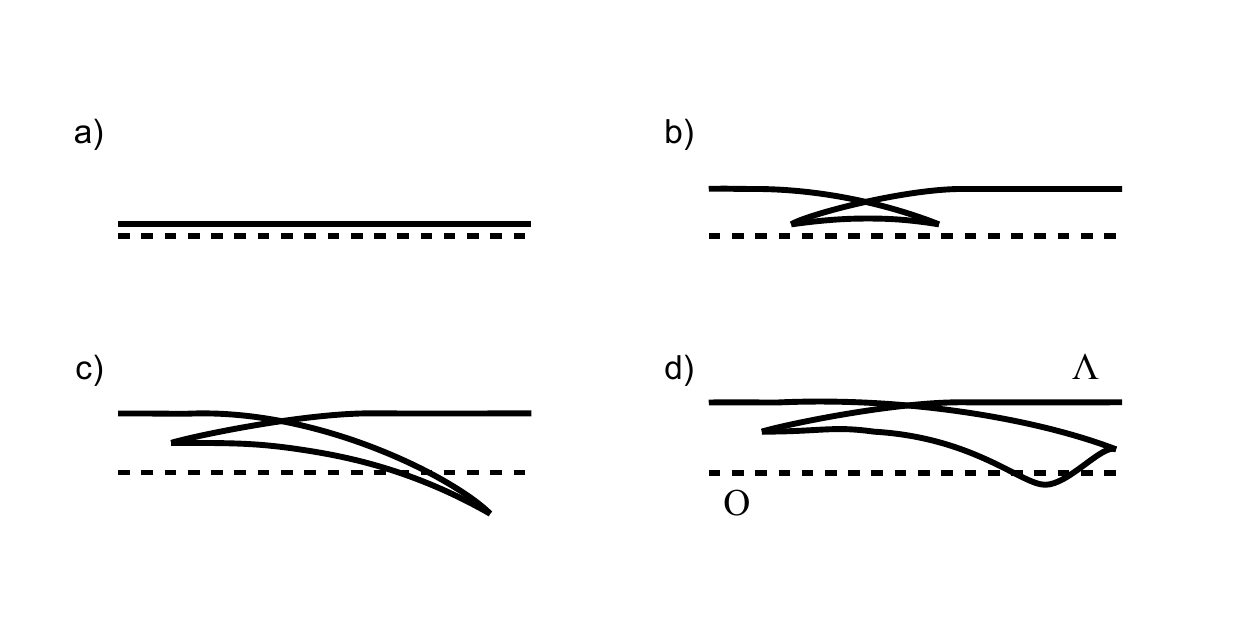}
\caption{A negative critical value from a positive isotopy.} 
\label{PosIso}
\end{figure}

The Legendrian link $(\mathrm{O},\Lambda)$ is a d\'ej\`a vu link in~$\Jet^1(S^1)$
for a purely topological reason. Indeed, suppose not. Then $\Lambda$ is 
Legendrian isotopic to $j^1(1)$ in the complement of $\mathrm{O}$ by Corollary~\ref{NonDjV}.
The two knots are however not even homotopic there because the winding number 
of $\Lambda$ around $\mathrm{O}$ defined as the degree of the projection to the $(u,p)$-plane 
minus the origin is~$\pm 1$. 

\begin{figure}[ht]
\centering
\captionsetup{margin=-1cm}
\includegraphics[scale=0.6]{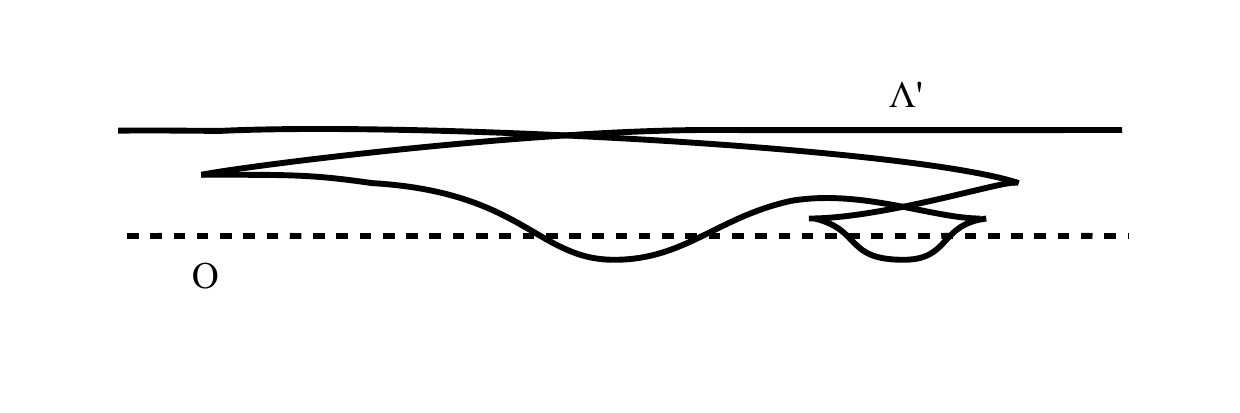}
\caption{D\'ej\`a vu with zero winding.} 
\label{PosIsoDeg0}
\end{figure}

The construction of $\Lambda$ can be modified to make its winding number 
around the zero section vanish. It is enough to create a pair of cusps 
during the isotopy from (c) to~(d), see Fig.~\ref{PosIsoDeg0}. 
The link $(\mathrm{O},\Lambda')$ is d\'ej\`a vu by Proposition~\ref{DjVdimn}
but one can also show that it is not even smoothly isotopic to $(\mathrm{O}, j^1(1))$. 

These two examples illustrate the following general fact based on
the results of Ding and Geiges~\cite{DG} and similar to~\cite[Theorem B]{ChNe1}.

\begin{prop}
\label{DjVS1Top}
A Legendrian link $(\Lambda_1,\Lambda_2)$ in $\Jet^1(S^1)$ such that its components 
are Legendrian isotopic to the zero section and $\Lambda_1\llcurly \Lambda_2$ 
is a d\'ej\`a vu link if and only if it is not smoothly isotopic to~$(\mathrm{O}, j^1(1))$.
\end{prop}

\begin{proof}
The `if' part follows from the definitions because as observed above 
a non d\'ej\`a vu link is even Legendrian isotopic to $(\mathrm{O}, j^1(1))$ 
by Corollary~\ref{NonDjV}.

A Legendrian link in $\Jet^1(S^1)$ such that its components are Legendrian isotopic
to the zero section is smoothly isotopic to $(\mathrm{O}, j^1(1))$
if and only if it is Legendrian isotopic to either $(\mathrm{O}, j^1(1))$
or $(j^1(1),\mathrm{O})$ by the main result of~\cite{DG}.

However, $(\Lambda_1,\Lambda_2)$ cannot be Legendrian isotopic to $(j^1(1),\mathrm{O})$
because $j^1(1)\ggcurly\mathrm{O}$ and the class of the zero section is orderable \cite{ChNe1,CFP}.
So if $(\Lambda_1,\Lambda_2)$ is smoothly isotopic to $(\mathrm{O}, j^1(1))$,
then it is Legendrian isotopic to it and is not d\'ej\`a vu,
which proves the `only if' part. 
\end{proof} 

This topological characterisation of condition (DjV) in the definition of 
d\'ej\`a vu links is specific to dimension~$3$, see Example~\ref{SmVsLeg}.


\subsection{D\'ej\`a vu Legendrian links in $\Jet^1(L)$}
\label{EmbJet}
\aftersubsec
In order to construct d\'ej\`a vu links in $\Jet^1(L)$
for an arbitrary closed manifold $L$ of dimension $n$, 
we `thicken' a given Legendrian in $\Jet^1(\R)$ 
to a similar Legendrian in $\Jet^1(\R^n)$ using 
the Legendrian suspension construction~\cite[\S 3.3]{ChNe3}.
We will only consider Legendrians in $\Jet^1(\R^n)$
that are equal to 1-jets of functions outside of
a compact set and Legendrian isotopies within
this class.
  
Let $\Lambda\subset\Jet^1(\R)$ be a Legendrian properly
isotopic to the zero section and equal to the 1-jet of
a bounded positive function outside of a compact subset.
Choose a positive Legendrian isotopy $\{\Lambda_t\}_{t\in [0,1]}$
in $\Jet^1(\R)$ such that $\Lambda_0=\Lambda$
and $\Lambda_1=j^1(C)$ for a positive constant~$C$.
(Such an isotopy can be constructed by taking any isotopy
from $\Lambda$ to $\mathrm{O}$ and composing it with 
appropriate positive shifts in the $u$-direction, see e.g.~\cite[\S 4]{CS}.)
Let $\chi:[0,+\infty)\to [0,1]$ be a smooth function such that
$$
\chi(t)=\left\{
\begin{array}{rl}
t& \text{ for } t<\eps\\
1&  \text{ for }t\ge 1-\eps
\end{array}
\right.
\quad\text{ and }\quad \chi'(t)>0 \text{ for } t< 1-\eps
$$
and consider the Legendrian family $\{\Lambda_{\chi(\|x\|^2)}\}_{x\in\R^{n-1}}$
with base $\R^{n-1}$, see~\cite[\S 3.1]{ChNe3}.
(Here $\|x\|$ denotes the Euclidean norm of $x\in\R^{n-1}$.) 
This family is transverse on $0<\|x\|^2<1-\eps$ and constant on $\|x\|^2\ge 1-\eps$.
Its Legendrian suspension is a Legendrian submanifold 
$$
\wt{\Lambda}\subset \Jet^1(\R^n)=\Jet^1(\R)\times T^*\R^{n-1}
$$
with the following properties:
\begin{enumerate}
\item $\wt{\Lambda}$ is the 1-jet of a positive function outside of a compact set~$K$.
\item Critical points of $\wt{\Lambda}$ in $K$ are of the form $((q,0),0,u)\in\Jet^1(\R^n)$,
where $(q,0,u)\in\Jet^1(\R)$ is a critical point of $\Lambda$.
\item If $\mathrm{O}\llcurly\Lambda$, then $\mathrm{O}\llcurly\wt{\Lambda}$ .
\end{enumerate}
By property~(1), $\wt{\Lambda}$ can be completed to a Legendrian in $\Jet^1(L)$ 
for any $n$-dimensional closed manifold $L$ so that no additional critical 
points in $\{u\le 0\}$ are created and property~(3) is preserved.

\begin{prop}
\label{DjVdimn}
Let $\Lambda$ be a Legendrian in $\Jet^1(\R)$ such that $\mathrm{O}\llcurly\Lambda$
and $\Lambda$ has $\kappa\ge 1$ transverse critical points in $\{u\le 0\}$ with the same
negative critical value. {\rm (\/}For instance, one may take the Legendrian $\Lambda'$
in Fig.~{\rm \ref{PosIsoDeg0}} and $\kappa=2$.{\rm )\/} 
Then $(\mathrm{O},\wt{\Lambda})$ is a d\'ej\`a vu link in~$\Jet^1(L)$.
\end{prop}

\begin{proof}
Any quadratic at infinity generating function $S$ for $\wt{\Lambda}$
will have $\kappa$ Morse critical points in $\{S\le 0\}$
with the same negative critical value. 
Therefore, the relative homology 
$\HH_*(S^0,S^{-\infty};\Z/2)$ will have $\kappa$ 
independent generators (in the degrees equal
to the indices of those critical points).

On the other hand, let $\{S_t\}_{t\in [0,1]}$ be a family 
of quadratic at infinity generating functions for 
a positive isotopy from $\mathrm{O}$ to $\wt{\Lambda}$. 
For small~$t$, the Legendrians in the isotopy
are 1-jets of positive functions and hence
all critical values of $S_t$ are positive 
and $\HH_*(S_t^0,S_t^{-\infty};\Z/2)=0$. 
So by Lemma~\ref{stability}, zero must be a critical
value of $S_\tau$ for some $\tau>0$ and the corresponding
Legendrian intersects $\mathrm{O}$, which shows that
the isotopy is not embedded.
\end{proof}

The argument in the proof uses the positivity of the
embedded isotopy only for small $t>0$. This agrees
with Lemma~\ref{ExistEmb}.

Spectral invariants are increasing along positive isotopies 
(see \cite[Proposition~2]{CFP} or \cite[Lemma~5.2]{ChNe1}) and hence cannot be used 
to detect the intersection of such an isotopy with $\mathrm{O}$. 
In other words, the image of the homomorphism 
$\HH_*(S^0,S^{-\infty};\bbk)\to \HH_*(L\times\R^N,S^{-\infty};\bbk)$
is trivial and therefore to invoke Lemma~\ref{stability}
we need to know that the kernel of this homomorphism is non-trivial.
This general principle may be stated in terms of 
the {\it finite bars\/} in the barcode of a generating function.

\begin{prop}
Let $\Lambda\subset\Jet^1(L)$ be a Legendrian submanifold
such that $\mathrm{O}\llcurly\Lambda$ and $\mathrm{O}\cap\Lambda=\varnothing$.
If for some {\rm (\/}and then any\/{\rm )} quadratic at infinity
generating Morse function of $\Lambda$ there is a finite bar
in its barcode containing $0\in\R$, then $(\mathrm{O},\Lambda)$
is a d\'ej\`a vu link.
\end{prop}

\begin{exm}[Smooth \textit{vs}\/ Legendrian links]
\label{SmVsLeg}
For $n\ge 2$ and $L=S^n$, a d\'ej\`a vu link $(\mathrm{O},\Lambda)$
{\it can\/} be smoothly isotopic to $(\mathrm{O},j^1(1))$ 
by an isotopy fixing $\mathrm{O}$.
The complement to the zero section in $\Jet^1(S^n)$ 
is diffeomorphic to $S^n\times (\R^{n+1}-\{0\})$
by the hodograph contactomorphism discussed in 
\S\ref{EmbST}. Taking $n\ge 2$, $k=0$ and $m=2n+1>2n-k$
in part~(b) of Haefliger's Th\'eor\`eme d'existence~\cite[p.~47]{H}, 
we see that embedded $n$-spheres in $S^n\times (\R^{n+1}-\{0\})$
are isotopic if and only if they are homotopic,
which can be inferred from the degrees of their projections to
$S^n$ and $\R^{n+1}-\{0\}$. For instance, $j^1(1)$ is
smoothly isotopic to the Legendrian $\wt{\Lambda'}\subset\Jet^1(S^n)$
obtained from the non-winding Legendrian $\Lambda'\subset\Jet^1(\R)$
shown in Fig.~\ref{PosIsoDeg0}. So the links $(\mathrm{O}, \wt{\Lambda'})$
and $(\mathrm{O},j^1(1))$ are smoothly isotopic in $\Jet^1(S^n)$ by an isotopy fixing~$\mathrm{O}$.
This example is similar to the example of smoothly 
unlinked but Legendrian linked $2$-spheres in~\cite[\S 6]{NT}.
\end{exm}

\subsection{D\'ej\`a vu Legendrian links in $ST^*\R^n$}
\label{EmbST}
\aftersubsec
The 1-jet bundle of the $(n-1)$-sphere is contactomorphic
to the co-sphere bundle of $\R^n$. Explicitly, 
let $\langle\cdot,\cdot\rangle$ denote the standard scalar product on $\R^n$
and let $S^{n-1}\subset\R^n$ be the unit sphere.
The map
$$
\R^n\times S^{n-1}\ni (x,q)\longmapsto\langle q,\cdot \rangle \in ST_x^*\R^n
$$
is a trivialisation of $ST^*\R^n$.
The {\it hodograph contactomorphism\/} is defined by the formula
$$
ST^*\R^n\ni(x,q) \longmapsto
(q,\langle x,\cdot \rangle|_{T_qS^{n-1}}, \langle x,q\rangle)\in\Jet^1(S^{n-1})
$$
The fibre of $ST^*\R^n$ over the origin is mapped to the zero section.

\begin{cor}
Legendrian d\'ej\`a vu links exist in the Legendrian isotopy class
of the fibre of $ST^*\R^n$.
\end{cor}

\begin{figure}[ht]
\centering
\captionsetup{margin=-1cm}
\includegraphics[scale=0.6]{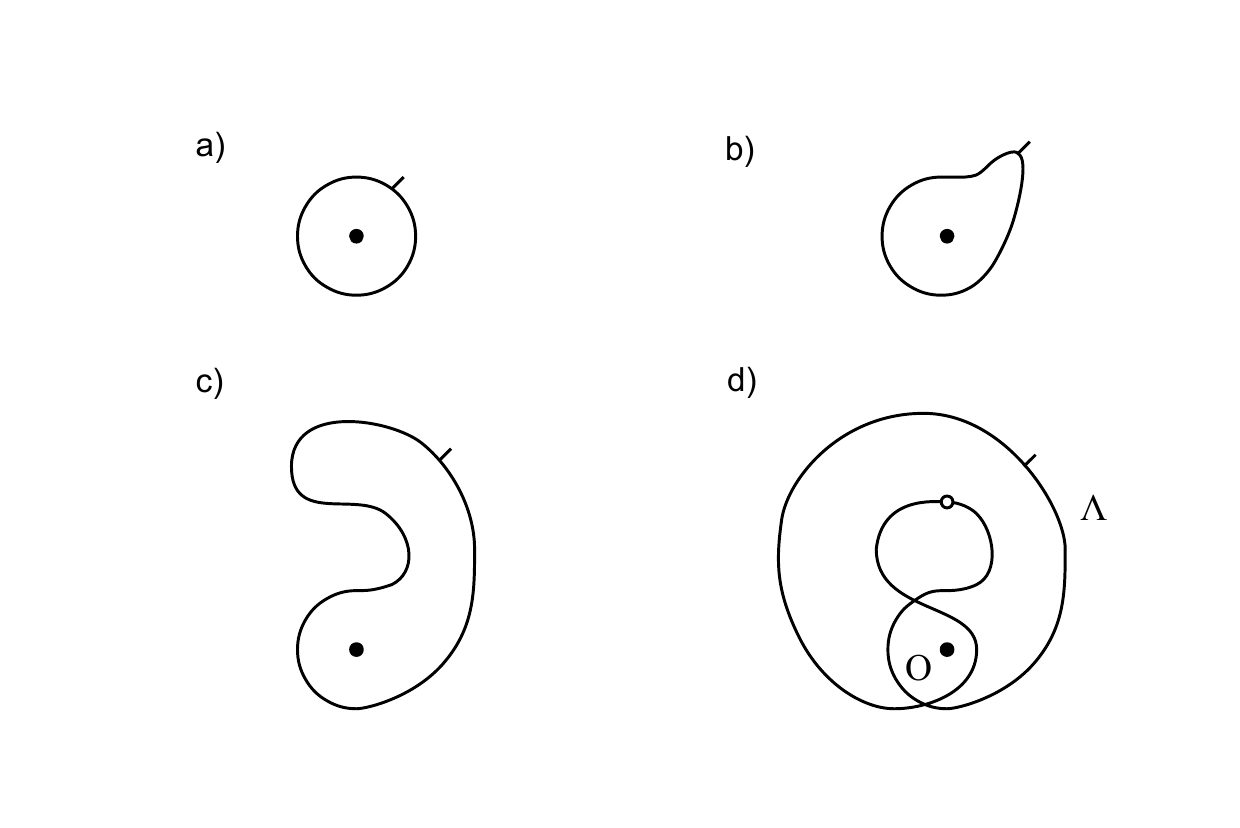}
\caption{D\'ej\`a vu Legendrian link in $ST^*\R^2$.} 
\label{PosIsoST}
\end{figure}

The hodograph pre-image in $ST^*\R^2$ of the Legendrian isotopy in Fig.~\ref{PosIso} 
is shown in Fig.~\ref{PosIsoST}. Here again we are drawing wavefronts,
which in this case are co-oriented projections of Legendrians to~$\R^2$.
The black dot represents the fibre over the origin. The critical point 
with negative $u$ in $\Jet^1(S^1)$ on Fig.~\ref{PosIso}(d) corresponds 
to the white dot on the wavefront of $\Lambda$ in~$\R^2$.
At this point, the co-orientation normal $q\in S^1$ 
is exactly opposite to the vector $x\in\R^2$ from the origin.
In this representation it is even more clear that the isotopy 
to $\Lambda$ is positive because the wavefronts in $\R^2$ 
are moving in the direction of their co-orientation, see~\cite[Example 2.2]{ChNe1}. 

\begin{figure}[ht]
\centering
\captionsetup{margin=-1cm}
\includegraphics[scale=0.6]{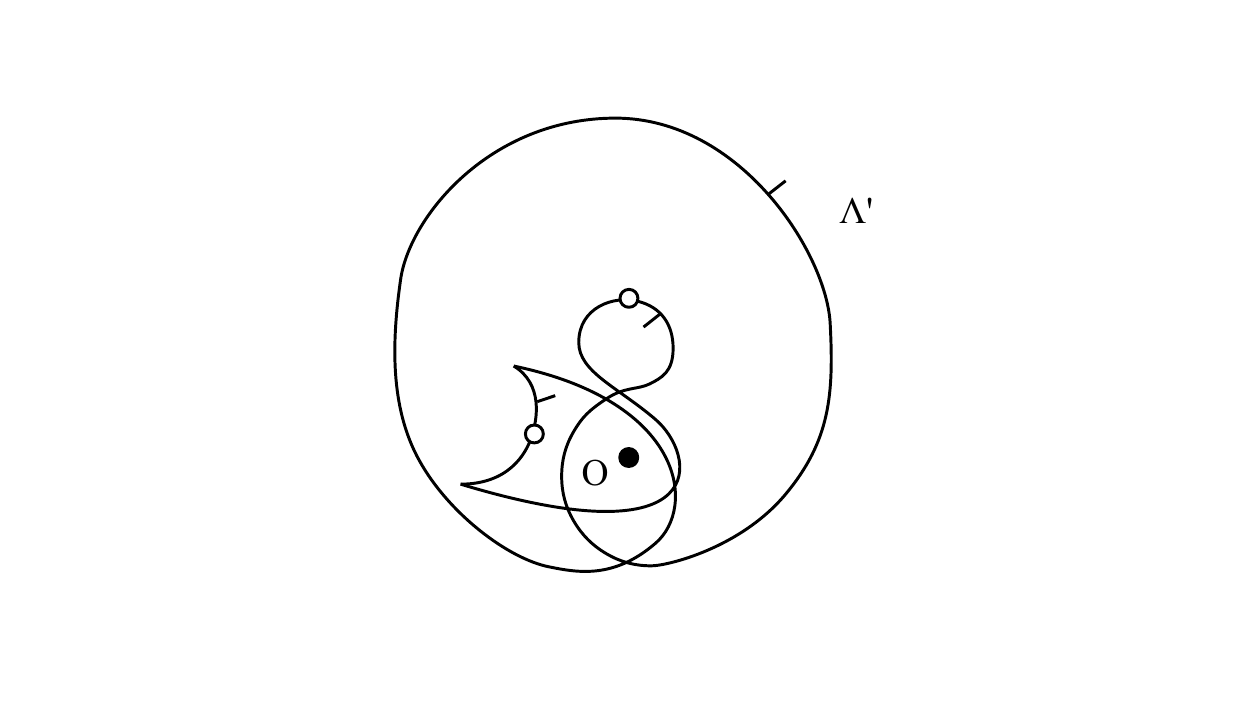}
\caption{D\'ej\`a vu Legendrian link without winding in $ST^*\R^{2n+1}$.} 
\label{PosIsoSTDeg0}
\end{figure}

Legendrians with the same properties in $ST^*\R^n$ for $n\ge 3$
can be obtained by applying the `finger move' on Fig.~\ref{PosIsoST} 
to the Legendrian whose wavefront in $\R^n$ is the outwardly co-oriented 
$(n-1)$-sphere around the origin. ({\it Warning\/}: These Legendrians 
will {\it not\/} be the hodograph images of the thickened Legendrians 
in $\Jet^1(S^{n-1})$ constructed by the suspension trick in \S\ref{EmbJet}.) 
Creating an additional $(n-2)$-sphere of cusps as on Fig.~\ref{PosIsoSTDeg0}
gives a d\'ej\`a vu link which for {\it odd\/} $n\ge 3$ is 
{\it smoothly\/} isotopic to the {\it non\/} d\'ej\`a vu link on Fig.~\ref{PosIsoST}(a) 
by the argument in Example~\ref{SmVsLeg}. For even $n$, one needs to create
two such `swallowtails', cf.~\cite[p.~257]{NT}.

\section{Lorentz geometry}

\subsection{Spacetimes and spaces of null geodesics}
\label{spacetimes}
\aftersubsec
A {\it spacetime\/} is a connected time-oriented Lorentz manifold $(\XX,\langle\text{ },\!\text{ }\rangle)$.
The Lorentz metric is taken to be of signature $(+,-,\dots,-)$ so that $\langle v,v\rangle>0$
for timelike vectors and $\langle v,v\rangle<0$ for spacelike vectors.
The time-orientation is a continuous choice of the future hemicone
$$
C^{\uparrow}_x\subset \{v\in T_x\XX\mid \langle v,v\rangle\ge 0, v\ne 0\}
$$
in the cone of non-spacelike vectors at each point~$x\in\XX$.
The vectors in $C^\uparrow_x$ are called {\it future-pointing}.
A piecewise smooth curve in $\XX$ is {\it future-directed\/}
if all its tangent vectors are future-pointing.

The {\it causality relation\/} $\le$  on $\XX$ is defined
by setting $x\le y$ if either $x=y$ or there is a future-directed
curve connecting $x$ to~$y$. 
The {\it chronology relation\/} $\ll$ is defined similarly
by writing $x\ll y$ if there is a future-directed timelike curve 
connecting $x$ to~$y$.

$\XX$ is {\it causal\/} if it does not contain closed 
future-directed curves. (This is equivalent to requiring
that $\le$ is a partial order.) A causal spacetime $\XX$ 
is {\it globally hyperbolic\/} if the causal interval $\{z\in\XX\mid x\le z\le y\}$ 
is compact for every~$x,y\in\XX$, see~\cite{BeSa2}
for the equivalence of this definition and the more classical one~\cite[Definition 5.24]{Pe} or~\cite[\S 6.6]{HaEl}. 
By the smooth splitting theorem~\cite{BeSa1}, 
a globally hyperbolic spacetime is foliated by smooth spacelike 
Cauchy (hyper)surfaces, where a {\it Cauchy surface\/} is a subset 
of a spacetime such that every endless future-directed curve 
intersects it exactly once.

The {\it space of null geodesics\/} of $\XX$ is the set $\mathfrak N_\XX$ 
of equivalence classes of endless future-directed null geodesics 
up to an orientation preserving affine reparametrisation~\cite{Lo2}. 
For a globally hyperbolic $\XX$ of dimension $\ge 3$, this space is a contact manifold 
and to every Cauchy surface $M\subset\XX$ there is associated 
a contactomorphism $\rho_M:\mathfrak N_\XX\overset{\cong}{\longrightarrow} ST^*M$
onto the co-sphere bundle of $M$, see e.g.\ \cite[pp.~252--253]{NT}
or \cite[\S\S 1-2]{ChNe4}.  

The set $\mathfrak S_x\subset\mathfrak N_\XX$ of all null geodesics 
passing through a point $x\in\XX$ is a Legendrian sphere in $\mathfrak N_\XX$
called the {\it sky\/} (or the {\it celestial sphere\/}) of that point.
For any Cauchy surface $M\subset\XX$ and a point $x\in M$, 
$\rho_M(\mathfrak S_x)=ST_x^*M$ and hence all skies are
mapped by $\rho_M$ to the Legendrian isotopy class of the fibre of~$ST^*M$, see \cite[\S 4]{ChNe1}.

The twistor map $x\mapsto\mathfrak S_x$ has the following properties
summarised in this form in~\cite[\S4.2]{ChNe5}.
If $x\le y$, then $\mathfrak S_x\lle\mathfrak S_y$, and if $x\ll y$, then
$\mathfrak S_x\llcurly\mathfrak S_y$, where $\lle$ and $\llcurly$
are the relations on Legendrians introduced in~\S\ref{LegIso}.
Moreover, if the Legendrian isotopy class of the fibre of $ST^*M$
is orderable (for instance, if $M$ is non-compact~\cite{ChNe2}), then 
the converse implications hold as well.

\subsection{D\'ej\`a vu moments}
\label{djvLorentz}
\aftersubsec
Let $\gamma=\gamma(t)$ be a future-directed timelike curve 
in a causal spacetime~$\XX$. A {\it d\'ej\`a vu moment\/} is
a point $x_+=\gamma(t_+)$ such that there exists a point
$x_-=\gamma(t_-)$ with $t_-<t_+$ and a null geodesic 
connecting $x_-$ to~$x_+$. (Causality implies that this null
geodesic will be future-directed.)
The curve $\gamma$ may be thought of as the world line
of an observer who is receiving the same light ray 
at inner time $t_+$ as at the time $t_-$ in the past.

\begin{prop}
\label{djvlinks2moments}
If the skies of two points in a globally hyperbolic spacetime~$\XX$ form a d\'ej\`a vu 
Legendrian link $(\mathfrak S_x,\mathfrak S_y)$, then there are d\'ej\`a vu 
moments on every future-directed timelike curve from $x$ to~$y$. 
\end{prop}

\begin{proof}
Such a curve $\gamma=\gamma(t)$, $t\in[0,1]$, 
defines a piecewise smooth isotopy of skies $\mathfrak S_{\gamma(t)}$
connecting $\mathfrak S_x$ to $\mathfrak S_y$
that is positive on its smooth segments by~\cite[Proposition 4.3]{ChNe5}.
Since there is no positive embedded isotopy between those skies,
it follows from Lemma~\ref{ExistEmb} that there exist $t_0\in(0,1]$ and $t_1\in[0,1)$
such that $\mathfrak S_x\cap\mathfrak S_{\gamma(t_0)}\ne\varnothing$
and $\mathfrak S_{\gamma(t_1)}\cap\mathfrak S_y\ne\varnothing$.
But this means exactly that there are null geodesics connecting 
$x$ to $\gamma(t_0)$ and $\gamma(t_1)$ to~$y$. The skies of $x$
and $y$ are disjoint, so they are not connected by a null geodesic.
Hence, we get at least two d\'ej\`a vu moments at $\gamma(t_0)$
and $y=\gamma(1)$.
\end{proof} 

\begin{rem}
The proof of Proposition~\ref{djvlinks2moments} shows, in other words, that every future-directed timelike curve from $x$ to $y$ 
must intersect the exponentiated future null cone of $x$ and the exponentiated
past null cone of $y$. 
\end{rem}

D\'ej\`a vu links in $ST^*\R^n$ of the types shown on Fig.~\ref{PosIsoST}(d) and Fig.~\ref{PosIsoSTDeg0} 
appear as links of skies in globally hyperbolic static spacetimes of the form $(\R^n\times\R, -g\oplus dt^2)$,
where $g$ is a `bumpy' Riemannian metric on $\R^n$ creating scattering obstacles. 
More advanced examples of this kind may be derived from~\cite[Fig.~5]{NT}. 

If the Legendrian isotopy class of skies
(i.e.\ the fibre class in $ST^*M$ for a Cauchy surface $M\subset\XX$) is {\it orderable}, 
then the proposition can be strengthened in two ways. 
First, there actually exists 
a future-directed timelike curve from $x$ to $y$ because 
$\mathfrak S_x\llcurly \mathfrak S_y$ implies $x\ll y$.
Secondly, intersections with the exponentiated null cones
exist for {\it every\/} curve connecting $x$ and $y$ 
by Proposition~\ref{DejavuOrdNoEmb}(ii). The first part
of the following example shows that without the orderability
assumption both assertions may be false.
 
\begin{exm}
Let $(M,g)$ be a Riemannian $Y^x_\ell$-manifold, which implies that
the fibre class in $ST^*M$ is {\it not\/} orderable, see \S\ref{yxl}.
The Lorentz direct product $(M\times\R,-g\oplus dt^2)$ is a globally
hyperbolic spacetime.

\smallskip
\noindent
{\bf (i)} If $M$ is not homeomorphic to the sphere, then the skies 
of two points $x=(\underline{x},0)$ and $y=(\underline{y},0)$ 
on the same Cauchy surface $M\times\{0\}$ form a d\'ej\`a vu link by Proposition~\ref{BSnondjv}. 
At the same time, the points are not causally related and can be connected by a (spacelike) curve 
inside $M\times\{0\}$ which does not intersect their exponentiated null cones.

\smallskip
\noindent
{\bf (ii)}  If $M$ is diffeomorphic to the sphere, then $(\mathfrak{S}_x,\mathfrak{S}_{y})$ 
is {\it not\/} a d\'ej\`a vu link. Let us consider the point $y'=(\underline{y},\ell)$.
If $\underline{y}$ is close to $\underline{x}$, then obviously $x\ll y'$. 
Note also that $\mathfrak{S}_{y'}=\mathfrak{S}_y$ because null geodesics in~$\XX$
are of the form $(\beta(s),s)$, where $\beta=\beta(s)$ is a naturally parametrised Riemannian geodesic in $(M,g)$. 
So $(\mathfrak{S}_x,\mathfrak{S}_{y'})$ is {\it not\/} a d\'ej\`a vu link. 
Nevertheless, the conclusion of Proposition~\ref{djvlinks2moments} 
is valid for $x$ and~$y'$ (non-vacuously because $x\ll y'$). To see this, 
note that $x\ll x'=(\underline{x},\ell)$ whereas $y'$ is causally unrelated to $x'$. 
Hence, every curve connecting $x$ and $y'$ intersects the boundary of $I^-(x')=\{z\in\XX\mid z\ll x'\}$. 
In a globally hyperbolic spacetime, this boundary is contained in the 
{\it closed\/} set $J^-(x')=\{z\in\XX\mid z\le x'\}$  
and is covered by null geodesics passing through $x'$ (and hence through~$x$) 
by~\cite[Proposition 3.71]{MS} and \cite[Proposition 2.20]{Pe}. 
It follows that every future-directed curve from $x$ to $y'$ 
must intersect the exponentiated future null cone of~$x$.
Applying the same argument to $I^+(y)=\{z\in\XX\mid y\ll z\}$ 
shows that such a curve must also intersect the exponentiated
past null cone of~$y'$.\qed
\end{exm}

The second part of the preceding example shows that the converse
to Proposition~\ref{djvlinks2moments} need not hold in general.
The example is very special, however, so one may venture the
following (optimistic) conjecture:

\begin{conje}
Let $\XX$ be a globally hyperbolic spacetime such that its Cauchy surface
is {\it not\/} homeomorphic to the sphere. If there are d\'ej\`a vu moments
on every future-directed timelike curve connecting two causally related points $x,y\in\XX$,
then their skies either intersect or form a d\'ej\`a vu Legendrian link
in~$\mathfrak{N}_\XX$.
\end{conje}

If this conjecture is true, it will have a number of purely
geometric consequences which may be used to support (or refute) it.
Assume, for instance, that $x$ is connected to $y$ by a future-directed 
timelike curve that does not intersect the exponentiated
future null cone of $x$. Then $\mathfrak{S}_x\llcurly\mathfrak{S}_y$
but the link $(\mathfrak{S}_x,\mathfrak{S}_y)$ is not d\'ej\`a vu
by Lemma~\ref{ExistEmb}. The conjecture would imply
that there exists a (maybe different) future-directed timelike curve from $x$ to~$y$
without d\'ej\`a vu moments and, in particular, not intersecting
the exponentiated past null cone of~$y$.

\end{document}